\documentclass[12pt]{amsart}

\usepackage[colorlinks=true,linkcolor=blue]{hyperref}
\usepackage{amsmath}
\usepackage{amssymb}
\usepackage{amsthm}

\usepackage[shortlabels]{enumitem}
\usepackage{graphicx}
\usepackage{epstopdf}
\usepackage{epsfig}

\usepackage{enumitem}
\usepackage{hyperref}

\usepackage{times}
\usepackage{relsize}
\usepackage{textcomp}
\usepackage{amssymb}
\usepackage[english]{babel}
\usepackage[autostyle]{csquotes}
\usepackage{epstopdf}
\usepackage{nicefrac}
\usepackage{tikz}
\theoremstyle{plain} 
\newtheorem{thm}{Theorem}[section]

\newtheorem{lemma}[thm]{Lemma}
\newtheorem{cor}[thm]{Corollary} 

\newtheorem{conj}[thm]{Conjecture}
\theoremstyle{remark}
\newtheorem{remark}[thm]{Remark}

\newtheorem{example}[thm]{Example}

\theoremstyle{definition}
\newtheorem{defin}[thm]{Definition}

\begin{document}

\title{A note on the factorization of iterated quadratics over finite fields}
\date{September 17, 2023}
\author{Vefa Goksel}
\address{Towson University, Towson, MD 21252, USA}
\email{vgoksel@towson.edu}

\begin{abstract} 
Let $f$ be a monic quadratic polynomial over a finite field of odd characteristic. In 2012, Boston and Jones constructed a Markov process based on the post-critical orbit of $f$, and conjectured that its limiting distribution explains the factorization of large iterates of $f$. Later on, Xia, Boston, and the author did extensive Magma computations and found some exceptional families of quadratics that do not seem to follow the original Markov model conjectured by Boston and Jones. They did this by empirically observing that certain factorization patterns predicted by the Boston-Jones model never seem to occur for these polynomials, and suggested a multi-step Markov model which takes these missing factorization patterns into account. In this note, we provide proofs for all these missing factorization patterns. These are the first provable results that explain why the original conjecture of Boston and Jones does not hold for \emph{all} monic quadratic polynomials.
\end{abstract}

\maketitle

\section{Introduction}
Let $f$ be a polynomial of degree $d\geq 2$ over a field $K$. Let $\bar{K}$ be an algebraic closure of $K$. For $n\geq 1$, we define the $n$-th iterate of $f$ by

\[f^n: = \underbrace{f\circ f\circ \cdots \circ f}_\text{$n$ times}.\]
We also make the convention that $f^0(x) = x$. One obtains an infinite $d$-ary rooted tree using the roots of $f^n(x)$ (in $\bar{K}$) for $n\geq 0$ as follows: For $n\geq 0$, we put roots of $f^n$ for the $n$th level of the tree, and we draw an edge between a root $\alpha$ of $f^n$ and a root $\beta$ of $f^{n+1}$ if $f(\beta) = \alpha$. If one further assumes that all iterates of $f$ are separable over $K$, then this becomes a complete infinite $d$-ary rooted tree. We denote this tree by $T_d$. Note that since any Galois element $\sigma \in \text{Gal}(\bar{K}/K)$ commutes with $f$, it preserves the connectivity relation on the tree, which yields to a continuous homomorphism 
\[\rho:\text{Gal}(\bar{K}/K) \rightarrow \text{Aut}(T_d).\]
The map $\rho$ is called an arboreal Galois representation. Describing the image of $\rho$ as a subgroup of the automorphism group $\text{Aut}(T_d)$ is one major open question in arithmetic dynamics.  There are natural comparisons between arboreal Galois representations and the well-established theory of $\ell$-adic Galois representations, which have inspired many works on arboreal Galois representations especially in the last two decades. We refer the reader to \cite{Jones14} and \cite[Section 5]{BIJMST19} for a general survey of the field. Also see \cite{Ahmad19,Bridy19,BDGHT21,Ferraguti19,FPC19} for a few examples of recent work on the subject. Based on the existing results, it is widely believed that im$(\rho)$ has a finite index inside of $\text{Aut}(T_d)$ (or in a smaller overgroup, see \cite{BDGHT21}) unless $f$ belongs to certain exceptional families of polynomials.

Let $K_n$ be the splitting field of $f^n$ over $K$. The image of $\rho$ can also be given as the inverse limit of the Galois groups $\text{Gal}(K_n/K)$. Suppose that $K$ is a number field. For any prime $\mathfrak{p}$ in the ring of integers $\mathcal{O}_K$, the image of the Frobenius class $\text{Frob}_\mathfrak{p}$ can be given by describing its image in each $\text{Gal}(K_n/K)$. Let $R_\mathfrak{p}:=\mathcal{O}_K/\mathfrak{p}$. Let $g\in \mathcal{O}_K[x]$ be a monic irreducible polynomial. Suppose that the reduced polynomial $\bar{g}\in R_{\mathfrak{p}}[x]$ is separable. It is well-known that the cycle structure of the action of $\text{Frob}_\mathfrak{p}$ on the roots of $g$ is given by the factorization type of the reduced polynomial $\bar{g}\in R_\mathfrak{p}[x]$. In particular, this is true for all but finitely many prime ideals $\mathfrak{p}$ in $\mathcal{O}_K$. Since $R_\mathfrak{p}$ is always a finite field, understanding images of Frobenius classes requires one to study the factorization of iterates of $f$ over finite fields.

This motivation led Boston and Jones \cite{JonBos12} to study factorizations of quadratic polynomials over finite fields, which is the simplest non-trivial case. Let $\mathbb{F}_q$ be a finite field of odd characteristic, and let $f\in \mathbb{F}_q[x]$ be a monic, irreducible quadratic. Boston and Jones constructed a Markov model to keep track of the factorization of iterates of $f$, and conjectured that the limiting distribution arising from their model explains the full factorization of large iterates of $f$. Later on, Xia, Boston, and the author \cite{GXB15} found new empirical data through Magma, which suggested that a more complicated model is required for some exceptional families of quadratic polynomials. Based on this empirical observation, they have made a multi-step Markov model, and conjectured that this new Markov model explains the factorization of large iterates of $f$ for any monic, irreducible quadratic $f\in \mathbb{F}_q[x]$.

The main motivation behind the multi-step Markov model constructed by Xia, Boston, and the author was some extensive Magma computations, which suggested that certain factorization patterns predicted by Boston-Jones model never occur for the exhibited exceptional families of polynomials. The goal of this paper is to \emph{prove} that those factorization patterns indeed never occur. Our work provides the first provable results which explain why the Boston-Jones model does not work for \emph{all} monic quadratic polynomials. In particular, it proves one of conjectures stated in \cite{GXB15} after a necessary correction (see Remark~\ref{rmk:about t=2 theorem}).

We also would like to note that besides its aforementioned connection to arboreal Galois representations, the factorization of composition of polynomials over finite fields is a well-studied subject in its own right. In addition to two articles already discussed, see also \cite{FGR16,FGR19,GomNic10,GOS14,HB-G19,OS10,Reis} for several examples of recent articles on the subject. See also Section 18 of the recent survey article \cite{BIJMST19} by Benedetto et. al. for a general overview of the subject. In fact, our main results address a case of \cite[Question 18.9]{BIJMST19}.

Before we state our main results, we first recall some notation and definitions. Let $f\in \mathbb{F}_q[x]$ be a monic quadratic, where $\mathbb{F}_q$ is a finite field of odd characteristic. Let $\gamma\in \mathbb{F}_q$ be the (unique) critical point of $f$. The \emph{post-critical orbit} of $f$ is given by the set
\[\mathcal{O}_f = \{f(\gamma),f^2(\gamma),\dots\}.\]
Since we work over a finite field, the post-critical orbit is necessarily finite. This is equivalent to saying that there exist some minimal integers $m\geq 0, n\geq 1$ such that $f^m(\gamma)=f^{m+n}(\gamma)$. In this case, we say that $f$ has \emph{orbit type} $(m,n)$. For notational convenience, we let $b_i:=f^i(\gamma)$ for $i\geq 1$. Hence, the post-critical orbit of $f$ becomes
\[\mathcal{O}_f = \{b_1,b_2,\dots,b_{m+n-1}\}.\]
\begin{remark}
\label{rem:orbit}
Many authors would define the orbit of a point by including the point itself. However, excluding the point will be convenient for our purposes in this paper, which is why we define the post-critical orbit as above.
\end{remark}

For simplicity, we will occasionally let $o_f:=|\mathcal{O}_f| = m+n-1$ in the paper. 

Boston and Jones defined the notion called \emph{$f$-type} of a monic, irreducible polynomial $g$ over $\mathbb{F}_q$, which helps one to keep track of the irreducible factors of $g(f^n(x))$ as $n$ gets large. We now recall the definition:
\begin{defin}
	\label{def:type}
Let $f\in \mathbb{F}_q[x]$ be a monic quadratic with orbit type $(m,n)$, and let $g\in \mathbb{F}_q[x]$ be any monic polynomial. We define the \emph{type} of $g$ at any $\beta\in \mathbb{F}_q$ to be $s$ if $g(\beta)$ is a square in $\mathbb{F}_q$, and $n$ if $g(\beta)$ is not a square in $\mathbb{F}_q$. We then define the \emph{$f$-type} of $g$ to be a string of length $m+n-1$, where the $k$-th entry of the string is the type of $g$ at $b_k$ (with the notation above).
\end{defin}
\begin{example}
	\label{ex:type}
Take $f=x^2+4x+6\in \mathbb{F}_7[x]$, and $g=x^3+x+1$. The critical point of $f$ is $\gamma=5\in \mathbb{F}_7$, hence we get $\mathcal{O}_f=\{2,4,1\}$. Since $g(2)=4$, $g(4)=6$, and $g(1)=3$, the $f$-type of $g$ is $snn$.
\end{example}
As mentioned earlier, the $f$-type of a polynomial $g$ is a useful tool that one can use to understand factorizations of $g(f^n(x))$ as $n$ gets large. For instance, for any monic, irreducible polynomial $g\in \mathbb{F}_q[x]$ with even degree, if the first $i$ entries of the $f$-type of $g$ are $n$, it follows that $g(f^j(x))$ is irreducible over $\mathbb{F}_q$ for all $j\leq i$ (see \cite[Lemma 2.5]{JonBos12}). On the other hand, if the first entry is $s$, it follows that $g(f(x))$ factors as a product of two irreducible polynomials of equal degree over $\mathbb{F}_q$ (see \cite[Proposition 2.6]{JonBos12}). See Section~\ref{sec:background} for more details on how the $f$-type of $g$ affects the factorizations of iterates $g(f^n(x))$.

We are now ready to state our main results.

\begin{thm}
\label{thm:t=2}
Let $f\in \mathbb{F}_q[x]$ be a monic quadratic with orbit type $(2,n)$ for some $n\geq 1$, and let $g\in \mathbb{F}_q[x]$ be any monic, irreducible polynomial of even degree whose $f$-type starts with $ns$. Then the type of each monic irreducible factor of $g(f^2(x))$ at $b_{n}$ is $s$.
\end{thm}
\begin{remark}
\label{rmk: explanation for theorem}
Since the orbit type of $f$ is $(2,n)$, evaluating $g(f^2(x))$ at $b_n$ gives
$g(f^2(b_n)) = g(b_{n+2}) = g(b_2)$, which is a square in $\mathbb{F}_q$ by hypothesis. On the other hand, again by hypothesis and using \cite[Proposition 2.6]{JonBos12}, the polynomial $g(f^2(x))$ factors as $h(x-\gamma)h(-(x-\gamma))$ for some monic irreducible polynomial $h\in \mathbb{F}_q[x]$. So, what the theorem is saying is that although the product $h(b_n-\gamma)h(-(b_n-\gamma))$ is a square in $\mathbb{F}_q$ by the hypothesis, it cannot happen that $h(b_n-\gamma)$ and $h(-(b_n-\gamma))$ are both non-squares.
\end{remark}
\begin{remark}
\label{rmk:about t=2 theorem}
Theorem~\ref{thm:t=2} proves one of conjectures in \cite{GXB15} (See Conjecture 4.5), after a minor error and a typo are corrected in the conjecture: The polynomials in the conjecture must be assumed to be monic, because the notion \emph{type} of a polynomial is not defined for non-monic polynomials, as indicated in the errata published by Boston and Jones \cite{JonBos20}. Secondly, one needs to look at the iteration $g(f^2(x))$ as opposed to $g(f(x))$. The reason for the latter is that $g(f(x))$ does not factor if $g$ has $f$-type $ns$, so the statement becomes uninteresting because it would be trivially true or trivially false only based on the $f$-type of $g$. 
\end{remark}
\begin{remark}
Quadratic polynomials over a finite field whose critical orbits have tail length $2$ have shown curious properties in other contexts as well. For instance, see \cite[Proposition 6.4]{Jones12}. The author thinks that it is worth investigating whether there is a link between these two occurrences of this class of polynomials.
\end{remark}
\begin{thm}
\label{thm:(3,1)}
Let $f\in \mathbb{F}_q[x]$ be a monic quadratic with orbit type $(3,1)$, and let $g\in \mathbb{F}_q[x]$ be any monic, irreducible polynomial of even degree. Suppose that $f$-type of $g$ starts with $nn$. If $H$ is any monic irreducible factor of $g(f^3(x))$, then the product $H(b_1)H(b_2)$ is a square in $\mathbb{F}_q$.
\end{thm}
\begin{remark}
\label{rmk:nnn vs nns}
If $f$-type of $g$ is $nnn$, then $g(f^3(x))$ is irreducible over $\mathbb{F}_q$ by Lemma 2.5 in \cite{JonBos12}. Since $f$ has orbit type $(3,1)$, one obtains $H(b_1) = g(f^3(b_1)) = g(f^3(b_2)) = H(b_2)$, which makes the implication in Theorem~\ref{thm:(3,1)} trivially true. Therefore, the interesting case is when $g$ has $f$-type $nns$.
\end{remark}
\begin{remark}
\label{rmk: about main results}
Theorem~\ref{thm:t=2} and Theorem~\ref{thm:(3,1)} together provide theoretical explanation for all the exceptional families of polynomials empirically observed in \cite{GXB15}. In particular, if $f$ is as in Theorem~\ref{thm:t=2} or Theorem~\ref{thm:(3,1)}, and $g\in\mathbb{F}_q[x]$ is any monic, irreducible polynomial of even degree, then $f$-types of irreducible factors of $g(f^2(x))$ (resp. $g(f^3(x))$) must belong to a more restricted set than what was assumed in the original Boston-Jones model \cite[Proposition 2.6]{JonBos12}. See Section~\ref{sec:proofs} (particularly Corollaries \ref{cor:missing_(2,n)} and \ref{cor:missing_(3,1)}) for more details on this. The nature of calculations leading to proofs seems to indicate that these exceptions may be results of some algebraic coincidences rather than a more general pattern. In this sense, our results can also be thought of as an evidence for the following conjecture, which had appeared in an earlier version of \cite{GXB15}.
\end{remark}
\begin{conj}
\label{conj:exceptions}
Let $f\in \mathbb{F}_q$ be a monic, quadratic polynomial, where $\mathbb{F}_q$ is a finite field of odd characteristic. The distribution of the factorization process of $f$ converges to limiting distributions arising from Boston-Jones model unless $f$ has orbit type $(3,1)$ or $(2,n)$ for some $n\geq 1$.
\end{conj}

The organization of the paper is as follows: In Section~\ref{sec:background}, we will provide some necessary background. In Section~\ref{sec:lemmas}, we will prove a few auxiliary lemmas. In Section~\ref{sec:proofs}, we will prove Theorem~\ref{thm:t=2} and Theorem~\ref{thm:(3,1)}.
\section{Background}
\label{sec:background}
In this section, we give some background, and recall definitions and results that we will use in the rest of the paper.

Let $f\in \mathbb{F}_q[x]$ be a monic quadratic with the unique critical point $\gamma\in \mathbb{F}_q$, where $\mathbb{F}_q$ is a finite field of odd characteristic. The \emph{type space} of $f$ is the set $S_f:=\{s,n\}^{o_f}$. For instance, if $f=x^2+1\in \mathbb{F}_3$, we have $S_f=\{ss,sn,ns,nn\}$ since $f$ has post-critical orbit size $2$. The polynomial $f$ has a natural action on the type space, as follows: If $t$ is a type, one obtains $f(t)$ by shifting each entry one position to the left, and using the former $k$th entry as the new final entry, where $k$ is such that $f^{o_f+1}(\gamma)=f^k(\gamma)$. Note that this action depends only on the orbit type of $f$. Let $g\in \mathbb{F}_q[x]$ be any irreducible polynomial of even degree. Any irreducible factor $h\in \mathbb{F}_q[x]$ of $g(f(x))\in \mathbb{F}_q[x]$ is called an \emph{immediate descendant} of $g$. If the $f$-type of $g$ starts with $n$, then $g(f(x))$ is irreducible over $\mathbb{F}_q$ (by \cite[Lemma 2.5]{JonBos12}), so $g$ has only one immediate descendant. On the other hand, if the $f$-type of $g$ starts with $s$, then $g(f(x))$ factors as the product of two irreducible polynomials of same degree in $\mathbb{F}_q[x]$ (by \cite[Proposition 2.6]{JonBos12}), hence $g$ has two immediate descendants in this case. As proved by Boston and Jones \cite[Proposition 3.4]{JonBos12}, some types cannot occur as immediate descendants of polynomials with certain types. The following lemma is a slightly different version of \cite[Proposition 3.4]{JonBos12} with our notation.
\begin{lemma}
\label{lem:allowable}
Let $f\in \mathbb{F}_q[x]$ be a monic quadratic with the critical point $\gamma$. Let $f$ have orbit type $(m,n)$. Suppose that $g\in \mathbb{F}_q$ is an irreducible polynomial of even degree such that $g(f^i(x))$ is separable for every $i\geq 1$. Let $a_1a_2\dots a_{m+n-1}$ be the $f$-type of an irreducible factor of $g(f^i(x)) \in\mathbb{F}_q[x]$ for some $i\geq 0$. If $a_1=s$, then $d_1\dots d_{m+n-1}$ cannot occur as an immediate descendant of a polynomial with $f$-type $a_1\dots a_{m+n-1}$ unless one of the following two conditions holds:
\begin{enumerate}
	\item $m=0$
	\item $m>0$ and $d_{m-1}d_{m+n-1} = a_m$.
\end{enumerate}
\end{lemma}
The following definition will also be crucial in the rest of the paper.
\begin{defin}
\label{def:n-step}
Let $f\in \mathbb{F}_q[x]$ be a monic quadratic, and let $g\in \mathbb{F}_q[x]$ be a monic polynomial of even degree. Let $k\geq 2$ be an integer. For each $0\leq i\leq k-1$, let $g(f^i(x))\in \mathbb{F}_q[x]$ have $t_i$ many irreducible factors with $f$-types $a_{i1},a_{i2},\cdots,a_{it_i}$. Then we call the chain
\[a_{01}/a_{02}/\cdots/a_{0t_0}\rightarrow a_{11}/a_{12}/\cdots/a_{1t_1}\rightarrow \cdots \rightarrow a_{(k-1)1}/a_{(k-1)2}/\cdots/a_{(k-1)t_{k-1}}\]
a \emph{$(k-1)$-step transition}.
\end{defin}
Boston and Jones constructed the Markov model for factorization of iterates of monic quadratic polynomials based on the assumption that if a type satisfies one of the conditions listed in Lemma~\ref{lem:allowable}, then it would always occur as an immediate descendant. However, Goksel, Xia, and Boston \cite{GXB15} later noticed through Magma computations that certain $2$-step and $3$-step transitions never occur for some exceptional families of polynomials $f$, although they satisfy one of the conditions in Lemma~\ref{lem:allowable}. Our goal in the remaining sections will be to \emph{prove} that these transitions indeed never occur for the exceptional families of quadratics exhibited in \cite{GXB15}.
\section{Auxiliary lemmas}
\label{sec:lemmas}
For any $F\in \overline{\mathbb{F}}_q[x]$, we let $V(F)$ be the vanishing set of $F$ over $\overline{\mathbb{F}}_q$. The following lemma together with its corollary will be important in the proofs of Theorem~\ref{thm:t=2} and Theorem~\ref{thm:(3,1)}. In particular, they will together be used to justify that certain expressions lie in $\mathbb{F}_q$, which will be crucial in proving Theorem~\ref{thm:t=2} and Theorem~\ref{thm:(3,1)}. The first part of the lemma is already given in \cite[Proposition 2.6]{JonBos12}, but we restate here for the convenience of the reader since it is used in the remaining parts of the lemma.
\begin{lemma}
\label{lem:GalConj}
Let $f\in \mathbb{F}_q[x]$ be any monic quadratic with the critical point $\gamma\in \mathbb{F}_q$, and let $g\in \mathbb{F}_q[x]$ be any irreducible polynomial. Suppose that $g(f^i(x))\in \mathbb{F}_q[x]$ is separable for every $i\geq 1$. For some $i\geq 1$, assume that $g(f^i(x))\in \mathbb{F}_q[x]$ is irreducible, and $g(f^{i+1}(x))\in \mathbb{F}_q[x]$ is reducible. Then the following three statements are true:

\begin{enumerate}[(a)]
\item $g(f^{i+1}(x)) = Ch(x-\gamma)h(-(x-\gamma))$ for a monic, irreducible polynomial $h\in \mathbb{F}_q[x]$, where $C\in \mathbb{F}_q$ is the leading coefficient of $g\in \mathbb{F}_q[x]$.
\item Let $\alpha_1,\alpha_2,\dots, \alpha_k, 2\gamma-\alpha_1,2\gamma-\alpha_2,\dots, 2\gamma-\alpha_k\in \overline{\mathbb{F}}_q$ be the (distinct) roots of $g(f(x))$. For $1\leq j\leq k$, define the sets $A_j$ and $B_j$ by
\[A_j:= V(h(x-\gamma))\cap V(f^i-\alpha_j).\]
and
\[B_j = V(h(x-\gamma))\cap V(f^i-2\gamma+\alpha_j).\]
Let $A_{f,g}^{(i)}:=\{A_1,A_2,\dots,A_k\}$ and $B_{f,g}^{(i)}:=\{B_1,B_2,\dots,B_k\}$. Set $C_{f,g}^{(i)} = A_{f,g}^{(i)}\cup B_{f,g}^{(i)}$. Any Galois element $\sigma\in \text{Gal}(\overline{\mathbb{F}}_q/\mathbb{F}_q)$ naturally induces a map $\tau_\sigma: C_{f,g}^{(i)} \rightarrow C_{f,g}^{(i)}$, which defines a transitive group action on the set $C_{f,g}^{(i)}$.
\item With the notation in part (b), for any $\sigma\in \text{Gal}(\overline{\mathbb{F}}_q/\mathbb{F}_q)$ and $1\leq j_1,j_2\leq k$, we have $\tau_\sigma(A_{j_1}) = A_{j_2}$ if and only if $\tau_{\sigma}(B_{j_1}) = B_{j_2}$.
\end{enumerate}
\end{lemma}
\begin{proof}
\begin{enumerate}[(a)]
\item See \cite[Proposition 2.6]{JonBos12}.
\item Note that by definition of $f$, it immediately follows that $\beta \in V(f^i-\alpha_j)$ if and only if $2\gamma-\beta \in V(f^i-\alpha_j)$ for any $1\leq j\leq k$. By definition of $h$, then, this forces exactly half of the elements of $V(f^i-\alpha_j)$ to lie in $V(h(x-\gamma))$, as otherwise it would contradict the separability assumption on $g(f^{i+1}(x))$. Similarly, exactly half of the elements of $V(f^i-2\gamma+\alpha_j)$ lies in $V(h(x-\gamma))$. Hence, for $1\leq j\leq k$, we can let $A_j:=\{\beta_{j1},\beta_{j2}\dots,\beta_{j\ell}\}$ and $B_j = \{\beta_{j1}', \beta_{j2}', \dots, \beta_{j\ell}'\}$, where $\ell=2^{i-1}$. For any $\sigma\in \text{Gal}(\overline{\mathbb{F}}_q/\mathbb{F}_q)$, define the map $\tau_\sigma:C_{f,g}^{(i)}\rightarrow C_{f,g}^{(i)}$ as follows: For $X,Y\in C_{f,g}^{(i)}$, if $\sigma(\theta)=\zeta$ for some $\theta\in X$ and $\zeta\in Y$, then we set $\tau_\sigma(X) = Y$. To prove that $\tau_\sigma$ gives a group action on the set $C_{f,g}^{(i)}$, the only non-trivial matter to check is to show that $\tau_\sigma$ is well-defined, as other conditions of a group action would then immediately follow from the action of $\sigma$ on the roots of $h(x-\gamma)$.

\textbf{Claim:} The map $\tau_\sigma$ is well defined. To prove this claim, consider a Galois element $\sigma\in \text{Gal}(\overline{\mathbb{F}}_q/\mathbb{F}_q)$, any $X\in C_{f,g}^{(i)}$, and two arbitrary elements $\beta,\theta\in X$. Suppose that $\sigma(\beta) = y\in Y$ for some $Y\in C_{f,g}^{(i)}$. By definition, we have
\[f^i(\beta) = \alpha, \text{ }f^i(\sigma(\beta)) = \alpha'\]
for two roots $\alpha, \alpha'$ of $g(f(x))$. Since $f$ is defined over the base field $\mathbb{F}_q$, $\sigma$ commutes with $f$, hence these two equalities immediately yield
\begin{equation}
\label{eq:alpha_j'}
\sigma(\alpha) = \alpha'.
\end{equation}
Now let $\sigma(\theta)=z\in Z$ for some $Z\in C_{f,g}^{(i)}$. This yields
\[f^i(\theta)=\alpha,\text{ }f^i(\sigma(\theta))=\alpha''\]
for a root $\alpha''$ of $g(f(x))$. This similarly gives
\begin{equation}
\label{eq:alpha_j''}
\sigma(\alpha) = \alpha''.
\end{equation}
Combining (\ref{eq:alpha_j'}) and (\ref{eq:alpha_j''}) implies $\alpha' = \alpha''$, which proves the equality $Y=Z$ by definitions of $Y$ and $Z$. Hence, the map $\tau_\sigma$ is well-defined. This finishes the proof that $\tau_\sigma$ defines a group action on $C_{f,g}^{(i)}$. Finally, note that this action is transitive because the action of $\text{Gal}(\overline{\mathbb{F}}_q/\mathbb{F}_q))$ on the roots of $h(x-\gamma)$ is transitive, completing the proof of part (b).
\item For some $\sigma\in \text{Gal}(\overline{\mathbb{F}}_q/\mathbb{F}_q)$ and $\beta\in A_{j_1}, \beta'\in A_{j_2}$, let $\tau_\sigma(\beta) = \beta'.$ By the proof of part (b), this is equivalent to $\sigma(\alpha_{j_1}) = \alpha_{j_2}$. Similarly, if $\tau_{\sigma}(\omega) = \omega'$ for some $\omega\in B_{j_1}, \omega'\in B_{j_3}$, we have $\sigma(2\gamma-\alpha_{j_1}) = 2\gamma-\alpha_{j_3}$, which is equivalent to $\sigma(\alpha_{j_1}) = \alpha_{j_3}$ since $\gamma\in \mathbb{F}_q$. This shows $j_2=j_3$ because $f(g(x))$ is separable over $\mathbb{F}_q$, which immediately proves the implication
\[\tau_\sigma(A_{j_1}) = A_{j_2} \implies \pi_{\sigma}(B_{j_1}) = B_{j_2}.\]
The other direction follows very similarly, hence we are done.
\end{enumerate}
\end{proof}
The following corollary of Lemma~\ref{lem:GalConj} will be useful in justifying that certain expressions which will arise in the proofs of Theorem~\ref{thm:t=2} and Theorem~\ref{thm:(3,1)} lie in the base field $\mathbb{F}_q$.
\begin{cor}
\label{cor:Sym_Coef}
Assume the notation in Lemma~\ref{lem:GalConj}. For $1\leq j\leq k$, let $A_j:=\{\beta_{j1},\beta_{j2},\dots,\beta_{j\ell}\}$ and $B_j=\{\omega_{j1}, \omega_{j2},\dots, \omega_{j\ell}\}$, where $\ell=2^{i-1}$. Let $S:=S(x_1,x_2,\dots,x_{\ell})$ be a symmetric function in $\ell$ variables. Define the polynomials $P_{f,S}^{(i)}$ and $Q_{f,S}^{(i)}$ by
\[P_{f,S}^{(i)}(x):=\prod_{j=1}^{k} (x-S(\beta_{j1},\beta_{j2},\dots,\beta_{j\ell})-S(\omega_{j1},\omega_{j2},\dots,\omega_{j\ell}))\]
and
\[Q_{f,S}^{(i)}(x):=\prod_{j=1}^{k} (x-S(\beta_{j1},\beta_{j2},\dots,\beta_{j\ell}))(x-S(\omega_{j1},\omega_{j2},\dots,\omega_{j\ell})).\]
Then $P_{f,S}^{(i)}$ and $Q_{f,S}^{(i)}$ are both defined over $\mathbb{F}_q$.
\end{cor}
\begin{proof}
Since $S$ is a symmetric function, by the proof of part (b) of Lemma~\ref{lem:GalConj}, and also using part (c) of Lemma~\ref{lem:GalConj}, it immediately follows that any Galois element $\sigma\in \text{Gal}(\overline{\mathbb{F}}_q/\mathbb{F}_q)$ permutes the roots of both $P_{f,S}^{(i)}$ and $Q_{f,S}^{(i)}$. Hence, the polynomials $P_{f,S}^{(i)}$ and $Q_{f,S}^{(i)}$ are fixed by the action of any $\sigma\in \text{Gal}(\overline{\mathbb{F}}_q/\mathbb{F}_q)$, which implies that they have all their coefficients in the base field $\mathbb{F}_q$, as desired.
\end{proof}
The following lemma gives some special identities satisfied by the post-critical orbit elements of $f$ when $f$ has orbit type $(2,n)$ and $(3,1)$, respectively. It will be a key lemma for the calculations of the next section, which will eventually lead to the proof that the corresponding expressions in Theorem~\ref{thm:t=2} and Theorem~\ref{thm:(3,1)} are squares in $\mathbb{F}_q$.
\begin{lemma}
\label{lem:identities}
Let $f\in \mathbb{F}_q[x]$ be a monic quadratic with critical point $\gamma\in \mathbb{F}_q$. Write $f(x)=(x-\gamma)^2+\gamma+c$ for some $c\in \mathbb{F}_q$. Then the following two statements are true.
\begin{enumerate}
\item Suppose that $f$ has orbit type $(2,n)$ for some $n\geq 1$, and let $\mathcal{O}_f=\{b_1,b_2,\dots,b_{n+1}\}$. Then $(b_n-\gamma)^2=-2c.$
\item Suppose that $f$ has orbit type $(3,1)$, and let $\mathcal{O}_f=\{b_1,b_2,b_3\}$. Then, we have \[(b_1-\gamma)^2(b_2-\gamma)^2 = 2(b_2-\gamma)\]
and
\[(b_1-\gamma)^2+(b_2-\gamma)^2 = -2c.\]
\end{enumerate}
\end{lemma}
\begin{proof}
\begin{enumerate}
	\item It follows from the fact that $b_{n+1} = 2\gamma - b_1$. The details are skipped.
	\item It follows from the fact that $(b_3-\gamma)=-(b_2-\gamma)$. The details are skipped.
\end{enumerate}
 \end{proof}
\section{Proofs of main results}
\label{sec:proofs}
We are now ready to prove Theorem~\ref{thm:t=2} and Theorem~\ref{thm:(3,1)}.
\begin{proof}[Proof of Theorem~\ref{thm:t=2}]
By part (a) of Lemma~\ref{lem:GalConj}, we have
\begin{equation}
\label{eq:h(x-gamma)}
g(f^2(x)) = h(x-\gamma)h(-(x-\gamma))
\end{equation}
for some monic, irreducible polynomial $h\in \mathbb{F}_q[x]$. Evaluating both sides of (\ref{eq:h(x-gamma)}) at $x=b_n$, since $f$ has orbit type $(2,n)$, we obtain
\[g(b_2) = h(b_n-\gamma)h(-(b_n-\gamma)).\]
Since the $f$-type of $g$ starts with $ns$, $g(b_2)$ is a square in $\mathbb{F}_q$, hence
\[h(b_n-\gamma)(\mathbb{F}_q^{\times})^2 = h(-(b_n-\gamma))(\mathbb{F}_q^{\times})^2.\] 
Thus, to prove the statement, it suffices to show that $h(b_n-\gamma)$ is a square in $\mathbb{F}_q$. For simplicity, we let $c_n:= b_n-\gamma$ for the rest of the proof.

Let $\alpha_1,\alpha_2,\dots,\alpha_k, 2\gamma-\alpha_1,2\gamma-\alpha_2,\dots,2\gamma-\alpha_k\in \overline{\mathbb{F}}_q$ be the (distinct) roots of $g(f(x))$. For $1\leq j\leq k$, let
\begin{equation}
\label{eq:intersection_singleton1}
V(h(x-\gamma))\cap V(f-\alpha_j)=\{\beta_{j1}\}
\end{equation}
and
\begin{equation}
\label{eq:intersection_singleton2}
V(h(x-\gamma))\cap V(f-2\gamma+\alpha_j) = \{\beta_{j2}\}.
\end{equation}
For $1\leq j\leq k$ and $i=1,2$, let $\theta_{ji}:= \beta_{ji}-\gamma$. Using (\ref{eq:intersection_singleton1}) and (\ref{eq:intersection_singleton2}), we immediately obtain
\begin{equation}
\label{eq:sum_square}
\theta_{j1}^2 + \theta_{j2}^2 = -2c.
\end{equation}
It also follows that we have
\begin{equation}
\label{eq:beta_j}
h(x) = \prod_{j=1}^{k} (x-\theta_{j1})(x-\theta_{j2}).
\end{equation}
Using (\ref{eq:beta_j}), we obtain
\begin{align*}
h(c_n) &= \prod_{j=1}^{k} (c_n-\theta_{j1})(c_n-\theta_{j2})\\
&= \prod_{j=1}^{k} (c_n^2-c_n(\theta_{j1}+\theta_{j2})+\theta_{j1}\theta_{j2})\\
&= \frac{1}{2^k}\prod_{j=1}^{k} (2c_n^2-2c_n(\theta_{j1}+\theta_{j2})+2\theta_{j1}\theta_{j2})\\
&= \frac{1}{2^k}\prod_{j=1}^{k} ((c_n-\theta_{j1}-\theta_{j2})^2+c_n^2-\theta_{j1}^2-\theta_{j2}^2)\\
&= \frac{1}{2^k} \left(\prod_{j=1}^{k} (c_n - \theta_{j1}-\theta_{j2})\right)^2,
\end{align*}
where we used the first part of Lemma~\ref{lem:identities} and (\ref{eq:sum_square}) in the last equality. We now let
\[F(x):= \prod_{j=1}^{k} (x-\theta_{j1}-\theta_{j2}) = \prod_{j=1}^{k} (x-(\beta_{j1}+\beta_{j2})+\gamma).\] 
 By taking $i=1$, $S=x$ for the polynomial $P_{f,S}^{(i)}$ given in Corollary~\ref{cor:Sym_Coef}, and noting that $\gamma\in \mathbb{F}_q$, it follows that the coefficients of $F$ lie in $\mathbb{F}_q$. Since $g$ is a polynomial of even degree, it also follows that $k$ is even, which shows that $h(c_n)=h(b_n-\gamma)$ is a square in $\mathbb{F}_q$, as desired.
\end{proof}
Theorem~\ref{cor:missing_(2,n)} gives a proof for the missing $2$-step transitions that were observed through Magma in \cite{GXB15}. We give this in the next corollary.
\begin{cor}
\label{cor:missing_(2,n)}
Let $f\in \mathbb{F}_q[x]$ be a monic quadratic with orbit type $(2,n)$. Let $g$ be any irreducible factor of $f^i$ for some $i\geq 1$. Suppose that $g$ has $f$-type starting with $ns$. Then the $2$-step transitions of the form
\[ns\dots \rightarrow s\dots \rightarrow \cdots nn/\cdots nn,\]
\[ns\dots \rightarrow s\dots \rightarrow \cdots ns/\cdots nn\]
and
\[ns\dots \rightarrow s\dots \rightarrow \cdots ns/\cdots ns\]
never occur for $g$. 
\end{cor}
\begin{proof}
By \cite[Proposition 2.6]{JonBos12}, any irreducible factor $g$ of $f^i$ has even degree. Therefore, in any $2$-step transition, Theorem~\ref{thm:t=2} implies that the penultimate entry for the $f$-type of any irreducible factor of $g(f^2(x))$ cannot be $n$, completing the proof of Corollary~\ref{cor:missing_(2,n)}.
\end{proof}
\begin{remark}
\label{rmk:missing2step_half}
Let $g$ be as in Corollary~\ref{cor:missing_(2,n)}. One can use Lemma~\ref{lem:allowable} and Corollary~\ref{cor:missing_(2,n)} to immediately see that precisely half of the $2$-step transitions of $g$ which are assumed to occur in Boston-Jones model never occurs. This provides an explanation for the discrepancy between the data arising from Boston-Jones model and the actual factorization data for iterates of corresponding family of quadratic polynomials. 
\end{remark}
\begin{proof}[Proof of Theorem~\ref{thm:(3,1)}]
Using Remark~\ref{rmk:nnn vs nns}, we can assume without loss that $g$ has $f$-type $nns$. By part (a) of Lemma~\ref{lem:GalConj}, then, we have
\begin{equation}
\label{eq:g(f^3)}
g(f^3(x)) = h(x-\gamma)h(-(x-\gamma))
\end{equation}
for some monic, irreducible polynomial $h\in \mathbb{F}_q[x]$.
Since $f$ has orbit type $(3,1)$, if we evaluate both sides of (\ref{eq:g(f^3)}) at $x=b_1$ and $x=b_2$, we obtain
\begin{equation}
\label{eq:x=b_1}
g(b_3) = h(b_1-\gamma)h(-(b_1-\gamma))
\end{equation}
and
\begin{equation}
\label{eq:x=b_2}
g(b_3) = h(b_2-\gamma)h(-(b_2-\gamma)).
\end{equation}
Hence, (\ref{eq:x=b_1}) and (\ref{eq:x=b_2}) together yield
\[h(b_1-\gamma)h(b_2-\gamma)(\mathbb{F}_q^{\times})^2 = h(-(b_1-\gamma))h(-(b_2-\gamma))(\mathbb{F}_q^{\times})^2.\]
Therefore, to prove the statement, it suffices to show that the product $h(b_1-\gamma)h(b_2-\gamma)$ is a square in $\mathbb{F}_q$. For simplicity, we let $c_i:=b_i-\gamma$ for $i=1,2$ for the rest of the proof.

Let $\alpha_1,\alpha_2,\dots,\alpha_k,2\gamma-\alpha_1,2\gamma-\alpha_2,\dots,2\gamma-\alpha_k\in \overline{\mathbb{F}}_q$ be the (distinct) roots of $g(f(x))$. For $1\leq j\leq k$, let
\begin{equation}
\label{eq:intersection_2element}
V(h(x-\gamma))\cap V(f^2-\alpha_j)=\{\beta_{j1},\beta_{j2}\}
\end{equation}
and
\begin{equation}
\label{eq:intersection_2element2}
V(h(x-\gamma))\cap V(f^2-2\gamma+\alpha_j) = \{\beta_{j3},\beta_{j4}\}.
\end{equation}
For $1\leq j\leq k$ and $i=1,2,3,4$, let $\theta_{ji}:=\beta_{ji}-\gamma$. The equations (\ref{eq:intersection_2element}) and (\ref{eq:intersection_2element2}) immediately yield 
\begin{equation}
\label{eq:identities(3,1)}
\theta_{j1}^2+\theta_{j2}^2 = \theta_{j3}^2+\theta_{j4}^2 = -2c.
\end{equation}
It also follows that we have
\begin{equation}
\label{eq:h(x-gamma)(3,1)}
h(x) = \prod_{j=1}^{k} \left(\prod_{i=1}^{4}(x-\theta_{ji})\right).
\end{equation}
Our goal was to show that the product $h(c_1)h(c_2)$ is a square in $\mathbb{F}_q$. To that end, we will now study the products $(c_i-\theta_{j1})(c_i-\theta_{j2})$ and $(c_i-\theta_{j3})(c_i-\theta_{j4})$ for $1\leq j\leq k$ and $i=1,2$.

We have 
\small\begin{align}
(c_1-\theta_{j1})(c_1-\theta_{j2})
&= c_1^2 - c_1(\theta_{j1}+\theta_{j2})+\theta_{j1}\theta_{j2}\nonumber\\
&= \frac{1}{2} \big(2c_1^2 - 2c_1(\theta_{j1}+\theta_{j2})+2\theta_{j1}\theta_{j2}\big)\nonumber\\
&= \frac{1}{2}\bigg(\big(c_1-\theta_{j1}-\theta_{j2}\big)^2+c_1^2-\theta_{j1}^2-\theta_{j2}^2\bigg)\nonumber\\
\label{eq:prodct with b_1}
&= \frac{1}{2}\bigg(\big(c_1-\theta_{j1}-\theta_{j2}\big)^2 +c_2+c\bigg),
\end{align}
\normalsize where we used (\ref{eq:identities(3,1)}) in the last equality. A similar calculation will also yield
\begin{equation}
\label{eq:product with b_2}
(c_2-\theta_{j1})(c_2-\theta_{j2}) = \frac{1}{2}\bigg(\big(c_2-\theta_{j1}-\theta_{j2}\big)^2 +c_3+c\bigg).
\end{equation}
Since $f$ has orbit type $(3,1)$, we have $c_3 = -c_2$. Using this in (\ref{eq:prodct with b_1}), we obtain
\begin{align}
\nonumber(c_1-\theta_{j1})(c_1-\theta_{j2})&=\frac{1}{2}\bigg(\big(c_1-\theta_{j1}-\theta_{j2}\big)^2 +c_2+c\bigg)\\ \nonumber&= \frac{1}{2}\bigg(\big(c_1-\theta_{j1}-\theta_{j2}\big)^2 - c_3 + c\bigg)\\
\nonumber&= \frac{1}{2}\bigg(\big(c_1-\theta_{j1}-\theta_{j2}\big)^2 - c_2^2\bigg)\
\label{eq:squarediff_b_1}\\
&= \frac{1}{2}\big(c_1-c_2-\theta_{j1}-\theta_{j2}\big)\big(c_1+c_2-\theta_{j1}-\theta_{j2}\big).
\end{align}
A similar calculation using (\ref{eq:product with b_2}) also yields
\begin{equation}
\label{eq:squarediff_b_2}
(c_2-\theta_{j1})(c_2-\theta_{j2}) = \frac{1}{2}\big(c_2-c_1-\theta_{j1}-\theta_{j2}\big)\big(c_1+c_2-\theta_{j1}-\theta_{j2}\big).
\end{equation}
Because of (\ref{eq:identities(3,1)}), it is clear that the same calculations will hold for the products $(c_1-\theta_{j3})(c_1-\theta_{j4})$ and $(c_2-\theta_{j3})(c_2-\theta_{j4})$ as well. Hence, we also have the identities
\begin{equation}
\label{eq:b_1 with j3,j4}
(c_1-\theta_{j3})(c_1-\theta_{j4}) = \frac{1}{2}\big(c_1-c_2-\theta_{j3}-\theta_{j4}\big)\big(c_1+c_2-\theta_{j3}-\theta_{j4}\big)
\end{equation}
and
\begin{equation}
\label{eq:b2 with j3,j4}
(c_2-\theta_{j3})(c_2-\theta_{j4}) = \frac{1}{2}\big(c_2-c_1-\theta_{j3}-\theta_{j4}\big)\big(c_1+c_2-\theta_{j3}-\theta_{j4}\big).
\end{equation}
Let
\[F(x) := \prod_{j=1}^{k} (x-\theta_{j1}-\theta_{j2})(x-\theta_{j3}-\theta_{j4}).\]
Using (\ref{eq:squarediff_b_1}), (\ref{eq:squarediff_b_2}), (\ref{eq:b_1 with j3,j4}), and (\ref{eq:b2 with j3,j4}) in (\ref{eq:h(x-gamma)(3,1)}), we now obtain
\[h(c_1)h(c_2) = \frac{1}{2^{4k}}F(c_1-c_2)F(c_2-c_1)\big(F(c_1+c_2)\big)^2.\]
Let
\[G(x):=\prod_{j=1}^{k} (x-\theta_{j1}-\theta_{j2}) = \prod_{j=1}^{k} (x-\beta_{j1}-\beta_{j2}+2\gamma)\] 
and
\[H(x):=\prod_{j=1}^{k} (x-\theta_{j3}-\theta_{j4}) = \prod_{j=1}^{k} (x-\beta_{j3}-\beta_{j4}+2\gamma).\]
By taking $i=2, S=x+y$ for the polynomial $Q_{f,S}^{(i)}$ given in Corollary~\ref{cor:Sym_Coef}, and noting that $\gamma\in \mathbb{F}_q$, it follows that $F(x) = G(x)H(x)$ has coefficients in $\mathbb{F}_q$. Hence, to finish the proof, it suffices to show that $F(c_1-c_2)F(c_2-c_1)$ is a square in $\mathbb{F}_q$. To that end, we will now do some calculations related to this product.

We have
\begin{align}
(c_1-c_2-\theta_{j1}-\theta_{j2})(c_2-c_1-\theta_{j1}-\theta_{j2})\nonumber&= (\theta_{j1}+\theta_{j2})^2-(c_1-c_2)^2\\
\label{eq:2ab+2cd}
&= 2\big(c_1c_2+\theta_{j1}\theta_{j2}\big),
\end{align}
where we used the second part of Lemma~\ref{lem:identities} and (\ref{eq:identities(3,1)}) in the last equality. Similarly, we also obtain
\begin{equation}
\label{eq:2xy+2zt}
(c_1-c_2-\theta_{j3}-\theta_{j4})(c_2-c_1-\theta_{j3}-\theta_{j4}) =  2\big(c_1c_2+\theta_{j3}\theta_{j4}\big).
\end{equation}
Setting $A=F(c_1-c_2)F(c_2-c_1)$, and using (\ref{eq:2ab+2cd}) and (\ref{eq:2xy+2zt}) yield
\begin{align}
A \nonumber&= 2^{2k} \prod_{j=1}^{k} (c_1c_2+\theta_{j1}\theta_{j2})(c_1c_2+\theta_{j3}\theta_{j4})\\
\nonumber &= 2^{2k} \prod_{j=1}^{k} ((c_1c_2)^2+c_1c_2(\theta_{j1}\theta_{j2}+\theta_{j3}\theta_{j4})+\theta_{j1}\theta_{j2}\theta_{j3}\theta_{j4})\\
\nonumber &= 2^k \prod_{j=1}^{k} (2(c_1c_2)^2+2c_1c_2(\theta_{j1}\theta_{j2}+\theta_{j3}\theta_{j4})+2\theta_{j1}\theta_{j2}\theta_{j3}\theta_{j4})\\
\label{eq:A_detailed}
&= 2^k \prod_{j=1}^{k} \big( (c_1c_2+\theta_{j1}\theta_{j2}+\theta_{j3}\theta_{j4})^2+(c_1c_2)^2-(\theta_{j1}\theta_{j2})^2-(\theta_{j3}\theta_{j4})^2\big).
\end{align}
By definition, $\pm \theta_{j1}, \pm \theta_{j2}$ are the roots of the quartic equation
\[t^4+2ct^2+b_2-\alpha_j = 0,\]
hence $(\theta_{j1}\theta_{j2})^2 = b_2-\alpha_j$. Similarly, we also obtain $(\theta_{j3}\theta_{j4})^2 = b_2+\alpha_j-2\gamma$. If we use these two equalities in (\ref{eq:A_detailed}), it gives
\begin{align}
A \nonumber &= 2^k \prod_{j=1}^{k} \big( (c_1c_2+\theta_{j1}\theta_{j2}+\theta_{j3}\theta_{j4})^2+(c_1c_2)^2-2c_2\big)\\
\label{eq:last_A}
&= 2^k \left(\prod_{j=1}^{k} (c_1c_2+\theta_{j1}\theta_{j2}+\theta_{j3}\theta_{j4})\right)^2,
\end{align}
where we used the second part of Lemma~\ref{lem:identities} in the last equality. If we now set
\[T(x) = \prod_{j=1}^{k} (x+\theta_{j1}\theta_{j2}+\theta_{j3}\theta_{j4}),\]  we obtain
\[A = 2^k(T(c_1c_2))^2.\]
Taking $i=2, S_1=-x_1x_2, S_2=-y_1y_2$ in Corollary~\ref{cor:Sym_Coef}, and recalling again that $\gamma\in \mathbb{F}_q$, it immediately follows that $T$ has its coefficients in $\mathbb{F}_q$. Recalling that $k$ is even (since $g$ has even degree), it now follows that $A=F(c_1-c_2)(c_2-c_1)$ is a square in $\mathbb{F}_q$, completing the proof of Theorem~\ref{thm:(3,1)}.
\end{proof}
\begin{cor}
	\label{cor:missing_(3,1)}
	Let $f\in \mathbb{F}_q[x]$ be a monic, irreducible quadratic with orbit type $(3,1)$. Let $g$ be any irreducible factor of $f^i$ for some $i\geq 1$. Then the $3$-step transitions of the form
	\[nns \rightarrow nss \rightarrow sss \rightarrow nss/nss\]
	and
	\[nns \rightarrow nss \rightarrow sss \rightarrow snn/snn\]
	never occur for $g$. 
\end{cor}
\begin{proof}
By \cite[Proposition 2.6]{JonBos12}, any irreducible factor $g$ of $f^i$ has even degree. Considering the missing $3$-step transitions that we want to prove, we can assume without loss that $g(f^2(x))$ is irreducible, and $g(f^3(x))$ is reducible over $\mathbb{F}_q$. By the action of $f$ on the type space, it follows that any such $g$ must have $f$-type $nns$. Let $\mathcal{O}_f=\{b_1,b_2,b_3\}$.  Theorem~\ref{thm:(3,1)} implies that for any irreducible factor $H$ of $g(f^3(x))$, the type of $H$ at $b_1$ and $b_2$ are equal to each other, which shows that the $3$-step transitions given above can never occur, as desired.
\end{proof}
\begin{remark}
\label{rmk:missing_transitions(3,1)}
Let $g$ be as in Corollary~\ref{cor:missing_(3,1)}. Similar to the previous case, Lemma~\ref{lem:allowable} and Corollary~\ref{cor:missing_(3,1)} together imply that precisely half of the $3$-step transitions of $g$ which are assumed to occur in Boston-Jones model never occurs. This explains why the factorization data predicted by Boston-Jones model does not match the actual factorization data for iterates of corresponding family of quadratic polynomials.
\end{remark}

\end{document}